\documentclass[12pt,twoside]{amsart} 

\makeatletter
\def\@citecolor{blue}
\def\@urlcolor{blue}
\def\@linkcolor{blue}
\makeatother

\usepackage{graphicx}
\usepackage{soul}
\usepackage[myheadings]{fullpage}
\usepackage{lmodern}
\usepackage[T1]{fontenc}
\usepackage{amsmath,amsthm,amssymb,amsfonts,latexsym,mathrsfs}
\usepackage[all]{xy}
\usepackage{paralist}
\usepackage{color}
\usepackage{ifthen}
\usepackage{wrapfig}
\usepackage[normalem]{ulem}
\usepackage{dsfont}
 \newfont{\Times}{ptmr scaled 1500}

\makeatletter

\@addtoreset{equation}{section}
\def\theequation{\thesection.\@arabic \c@equation}
\def\@citecolor{blue}
\def\@urlcolor{blue}
\def\@linkcolor{blue}
\def\theenumi{\@roman\c@enumi}

\makeatother

\theoremstyle{plain}
\newtheorem{theorem}[equation]{Theorem}
\newtheorem{lemma}[equation]{Lemma}
\newtheorem{corollary}[equation]{Corollary}
\newtheorem{proposition}[equation]{Proposition}

\theoremstyle{definition}
\newtheorem{remark}[equation]{Remark}

\newtheorem{remarks}[equation]{Remarks}

\newtheorem{example}[equation]{Example}

\newtheorem{definition}[equation]{Definition}

\newtheorem{construction}[equation]{Construction}

\newcommand{\I}{\mbox{$\  \Longrightarrow \ $}}
\newcommand{\II}{\mbox{$\  \Longleftrightarrow \ $}}

\newcommand{\e}{\mbox{\sf  {$\ell$}}}

\newcommand{{\natsmall}}{\mbox{${\rm\scriptstyle I\hspace{-.6 mm}N}$}}

\newcommand{\enu} {\begin{enumerate}} 
\newcommand{\enua} {\begin{enumerate}[ $(a)$]} 
\newcommand{\enui} {\begin{enumerate}[ $i$.] } 
\newcommand{\denu} {\end{enumerate}}

\def\NZQ{\mathbb}               
\def\NN{{\NZQ N}}
\def\ZZ{{\NZQ Z}}
\def\frk{\mathfrak}               
\def\mm{{\frk m}}

\def\opn#1#2{\def#1{\operatorname{#2}}} 
\opn\depth{depth}
\opn\PF{PF}
\opn\codim{codim}
\newcommand{\RR}{R(I)_{a,b}}
\newcommand{\om}{\omega_R}
\newcommand{\du}{\! \Join^b \!}
\AtEndDocument{\bigskip\footnotesize
\noindent\textsc{Anna Oneto, - DIMA - University of Genova - Via Dodecaneso 35 - 16146 Genova - Italy} \\
\textit{E-mail address}: \texttt{oneto@dime.unige.it}} 

\AtEndDocument{\bigskip\footnotesize
\newline\noindent\textsc{Francesco Strazzanti - Department of Mathematics - University of Pisa - Largo Bruno Pontecorvo 5 - 56127 Pisa - Italy} \\
\textit{E-mail address}: \texttt{strazzanti@mail.dm.unipi.it}}

\AtEndDocument{\bigskip\footnotesize
\newline\noindent\textsc{Grazia Tamone - DIMA - University of Genova - Via Dodecaneso 35 - 16146 Genova - Italy} \\
\textit{E-mail address}: \texttt{tamone@dima.unige.it}}

\begin{document}
\title{One-dimensional Gorenstein local rings with decreasing Hilbert function}

\author{Anna Oneto, Francesco Strazzanti, and Grazia Tamone}

\begin{abstract}
\noindent In this paper we  solve a problem posed by M.E. Rossi:  {\it  Is the Hilbert function of a Gorenstein local ring of dimension one not decreasing? }  More precisely, for any integer $h>1$, $h \notin\{14+22k, \, 35+46k \ | \ k\in\NN\}$, we construct infinitely many one-dimensional Gorenstein local rings, included integral domains, reduced and non-reduced rings, whose Hilbert function decreases at level $h$; moreover we prove that there are no bounds to the decrease of the Hilbert function. The key tools are numerical semigroup theory, especially some necessary conditions to obtain decreasing Hilbert functions found by the first and the third author, and a construction developed by V. Barucci, M. D'Anna  and the second author, that gives  a family of quotients of the Rees algebra. Many examples are included.

\end{abstract}
\date{\today}
\maketitle

\section*{\bf Introduction}

 Given a one-dimensional Cohen-Macaulay local ring $(R,\mm,k)$, let $G$ be its associated graded ring $G=\oplus_{h\geq 0}\big( \mm^h/\mm^{h+1}\big)$ and $H_R$ be the Hilbert function of $R$, defined as $H_R(h)=H_G(h)=\dim_k\big( \mm^h
/\mm^{h+1}\big)$. 
The Cohen-Macaulayness of $G$ and the behaviour of the Hilbert function are  classic topics in local algebra.   Starting from 1970s with the basic results of J.D. Sally \cite{js}, \cite{js1}, \cite{js2}, many authors have contributed to these themes; for instance we recall J. Elias \cite{el}, M.E. Rossi and G. Valla \cite{rv} and Rossi's survey \cite {r}. 
It is well-known that if $G$ is Cohen-Macaulay, the function $H_R$ is non-decreasing. 
On the other hand, when $\depth(G)=0$, $H_R$ can decrease, i.e. $H_R(h-1)>H_R(h)$ for some $h$; in this case we say that $H_R$ decreases at level $h$ and that $R$ has decreasing Hilbert function. 
  When $R$ is Gorenstein, M.E. Rossi asked in \cite[Problem 4.9]{r} if $H_R$ is always non-decreasing and in the last decade several authors found partial positive answers to this problem, especially in the case of numerical semigroup rings:  \\
$\bullet$ In \cite{AM} F. Arslan and P. Mete, for large families of complete intersection rings and the Gorenstein numerical semigroup rings with embedding dimension 4, under some arithmetical conditions;\\ 
$\bullet$ In \cite{AMS} F. Arslan, P. Mete  and M. \c{S}ahin, for infinitely many families of Gorenstein rings obtained by   introducing  the notion of nice gluing of numerical semigroups;\\
$\bullet$ In \cite{PT} D.P. Patil and the third author, for the rings associated with balanced numerical semigroups with embedding dimension $4$; \\
$\bullet$ In \cite{ASS} F. Arslan, N. Sipahi  and N. \c{S}ahin,  for other 4-generated Gorenstein numerical semigroup rings constructed by  non-nice gluing; \\
$\bullet$ In \cite{JZ} R. Jafari and S. Zarzuela Armengou, for some families of numerical Gorenstein semigroup rings through the concept of extension; \\
$\bullet$ In \cite{AKN} F. Arslan, A. Katsabekis, and M. Nalbandiyan, for other families of Gorenstein 4-generated numerical semigroup rings; \\
$\bullet$ In \cite{OT} the first and the third author,   for numerical semigroup rings such that $\nu\geq e-4$, where $\nu$ and $e$ denote respectively the embedding dimension and the multiplicity of $R$. 
\par
In this paper we show that Rossi's problem  has  negative answer, by constructing, among others, explicit examples of Gorenstein numerical semigroup rings with decreasing Hilbert function. 
They are  particular rings of a family  introduced and studied by V. Barucci, M. D'Anna and the second author  in \cite{BDS} and \cite{BDS2} to provide a unified approach to Nagata's idealization and amalgamated duplication. Given a commutative ring $R$ and an ideal $I$, for any $a,b \in R$ the rings $\RR$ are defined as suitable quotients of the Rees algebra of $I$.
  These have many good properties, in particular, if $R$ is a one-dimensional local ring, so is $\RR$. If $R$ is Cohen-Macaulay, another important fact is that $\RR$ is Gorenstein if and only if $I$ is a canonical ideal of $R$; in this case, when $R$ is almost Gorenstein, we prove that the Hilbert function of $\RR$ depends only on the   Cohen-Macaulay type and the Hilbert function of $R$. 
The crucial result is that if $R$ is an almost Gorenstein ring with  $H_R(h-2)>H_R(h)$ for some $h \geq 3$ and $I$ is a canonical ideal of $R$, then $\RR$ is a one-dimensional Gorenstein local ring with Hilbert function decreasing at level $h$.  

We find such rings through numerical semigroup theory. A numerical semigroup $S$ is a submonoid of the natural numbers that has finite complement in $\NN$; if $S$ is generated by $s_0, \dots, s_{\nu-1}$ and $k$ is a field, then $k[[S]]:=k[[t^{s_0}, \dots, t^{s_{\nu-1}}]]$ is the numerical semigroup ring associated with $S$ and many properties of $k[[S]]$, such as the Hilbert function, are still contained in $S$. In this context  $k[[S]]$ is almost Gorenstein if and only if $S$ is a so-called almost symmetric semigroup.

Hence to achieve our results we   look  for almost symmetric semigroups with decreasing Hilbert function. If $e$ and $\nu$ are the multiplicity and embedding dimension of $S$ (or equivalently of $k[[S]]$) we  first show  that we need $e-v \geq 4$. By using a result of \cite{OT} and a theorem of H. Nari \cite{Nari} we give an explicit construction of a family of almost symmetric semigroups with the required properties. We also show other examples with the above properties. In conclusion we prove that for any integers $m \geq 1$ and $h > 1, h \notin\{ 14+22k, \, 35+46k \ | \ k\in\NN\}$, there exist infinitely many non-isomorphic one-dimensional Gorenstein local rings $R$ such that $H_R(h-1)-H_R(h) > m$; this class always contains numerical semigroup rings, non-reduced rings, and reduced rings that are not integral domains. 

We include several examples in the numerical semigroup case. In fact if $R$ is a numerical semigroup ring and $b=t^m$, with $m$ odd, the ring $R(I)_{0,-b}$ is isomorphic to the numerical semigroup ring associated with the numerical duplication, a construction introduced and studied by M. D'Anna and the second author in \cite{DS}.\\[2mm]
The structure of the paper is the following. In the first section we introduce the family $\RR$ and show how to reduce the problem to find a suitable almost Gorenstein ring, see Corollary \ref{almsym}.
In Section 2 we describe a procedure that gives infinitely many almost Gorenstein semigroup rings satisfying the desired properties, see Construction \ref{asd} and Theorem \ref{l1}.
In Section 3 we prove the main result, see Theorem \ref{corMax}; moreover we give explicit examples of one-dimensional Gorenstein local semigroup rings with decreasing Hilbert function and other interesting examples based on the above constructions; see e.g. Example \ref{es64} with Hilbert function $[1,53,54,54,53,53,56,59,61,63,64 \rightarrow]$, Example \ref{l15} with Hilbert function decreasing at many levels, and Example \ref{mol}  for a ring with smaller multiplicity and embedding dimension. Finally the  appendix contains the technical results  needed to prove Theorem \ref{l1}.

\section{ \bf  Reduction to the almost Gorenstein case}

Let $R$ be a commutative ring with identity and let $I$ be a proper ideal of $R$. The Rees algebra of $I$ is the ring $\mathcal{R}_+:=\oplus_{n \geq 0} I^n t^n \subseteq R[t]$. Consider the ideal $(t^2+at+b)$ of $R[t]$, where $a$ and $b$ are elements of $R$, and let $I^2(t^2+at+b)$ denote its contraction to the Rees algebra.
In \cite{BDS} it is introduced and studied the following family of rings
$$
\RR:=\frac{\mathcal{R}_+}{I^2(t^2+at+b)}.
$$
One aim of this construction was to provide a unified approach to Nagata's idealization (see \cite{AW}) and amalgamated duplication (see \cite{D} and \cite{DF}), which are isomorphic to $R(I)_{0,0}$ and $R(I)_{-1,0}$ respectively. In fact in \cite{BDS} and \cite{BDS2} it is proved that several properties of the family are independent of $a$ and $b$. In particular $R$ is local if and only if $\RR$ is local, $R$ and $\RR$ have the same dimension and $\RR$ is Cohen-Macaulay if and only if $R$ is Cohen-Macaulay and $I$ is a maximal Cohen-Macaulay module; further, in the last case   if $I$ contains a regular element,  $\RR$ is Gorenstein if and only if $I$ is a canonical ideal of $R$.  

We are interested in the Hilbert function of these rings. In \cite[Proposition 2.3]{BDS} it is stated that their Hilbert functions do not depend on $a$ and $b$, but actually in the proof it is shown more and then we restate that proposition in the version   we need:

\begin{proposition} \label{HilbFun}  
If $(R,\mm)$ is a local ring and $I$ is an ideal of $R$, then for any $h \geq 2$ $H_{\RR}(h)= H_R(h) + \ell_R \left(I\mm^{h-1}/I\mm^h\right)$,
where $\ell_R$ denotes the length as $R$-module. In particular it is independent of $a$ and $b$. 
\end{proposition}

We are interested in the case in which $R$ is a one-dimensional Cohen-Macaulay local ring and $I$ is a canonical ideal of $R$. In this case    we can easily compute $\ell_R (I\mm^{h-1}/I\mm^h)$, under an extra hypothesis on $R$: almost Gorensteinness. Following \cite{BF} and \cite{GMP} we recall the definition in the one-dimensional case:

\begin{definition}
Let $(R,\mm)$ be a one-dimensional Cohen-Macaulay local ring with a canonical module $\om$ such that $R \subseteq \om \subseteq \overline{R}$. Then $R$ is said to be {\em almost Gorenstein} if $\mm \om = \mm$.
\end{definition}

From now on we assume that $R$ is one-dimensional. 
In the setting of the previous definition, chosen a regular element $a \in R$ such that $a\om \subset R$, the ideal $I=a\om$ is a canonical ideal of $R$ and all canonical ideals of $R$ can be obtained in this way (see e.g. \cite[Corollary 2.8]{GMP}).\\  
If $R$ is an almost Gorenstein ring and $I$ is a canonical ideal of $R$, for any $h \geq 2$ we have 
\begin{equation} \label{lunghezza}
\ell_R \left(\frac{I\mm^{h-1}}{I\mm^h}\right)=\ell_R \left(\frac{a\om \mm^{h-1}}{a \om \mm^h}\right)=\ell_R \left(\frac{a \mm^{h-1}}{a \mm^h}\right)=\ell_R \left(\frac{\mm^{h-1}}{\mm^h}\right)=H_R (h-1).
\end{equation}

\noindent
Therefore we get the following:

\begin{proposition} \label{Formula}
Let $R$ be an almost Gorenstein ring and let $I$ be a canonical ideal of $R$. Let $\nu(R)$ and $t(R)$ denote the embedding dimension and the Cohen-Macaulay type of $R$ respectively. Then the Hilbert function of $\RR$ is:
\begin{equation*}
\begin{split}
H_{\RR}(0)&= 1 \\
H_{\RR}(1)&= \nu(R) + t(R) \\
H_{\RR}(h)&= H_R(h) + H_R(h-1) \ \ \ \ \ {\text if } \ h \geq 2.
\end{split}
\end{equation*}\end{proposition}

\begin{proof}
 By Proposition \ref{HilbFun} and (\ref{lunghezza}) we only need to show the statement about $H_{\RR}(1)$: actually this equality is true even if $R$ is not almost Gorenstein . In fact, since the minimal number of generators of a canonical module is the   Cohen-Macaulay  type of the ring (see e.g. \cite[Proposition 3.3.11]{BH}), it is easy to deduce the thesis from Proposition \ref{HilbFun}. 
\end{proof}

\begin{corollary} \label{almsym}
Let $R$ be an almost Gorenstein ring and let $I$ be a canonical ideal of $R$. Then $\RR$ is a Gorenstein ring for any choice of $a,b \in R$. Further $H_{\RR}(1)\!-\!H_{\RR}(2)=t(R)\!-\!H_R(2)$ \, and \,\, $H_{\RR}(h\!-\!1)\!-\!H_{\RR}(h)=H_R(h\!-\!2)\!-\!H_R(h)$ for any $h\geq 3$.
\end{corollary}

 \begin{proof}
Since $I$ is a canonical ideal, $\RR$ is a Gorenstein ring by \cite[Corollary 3.3]{BDS}. Moreover, by the previous proposition, for any $h \geq 3$ we have
$$
H_{\RR}(h-1)-H_{\RR}(h)= H_R(h-1) + H_R(h-2) - H_R(h) - H_R(h-1)= H_R(h-2) - H_R(h).
$$
The first formula can be found in the same way, since $H_R(1)=\nu(R)$.
\end{proof}

 \begin{remark} {\rm
In this paper we are interested in {\it negative results}, anyway it is clear that Proposition \ref{HilbFun} can be also used to get {\it positive results}. A local ring $S$ is said to have minimal multiplicity if its multiplicity is $1 + \codim S$. In \cite[Theorem 1.1]{P} it is proved that if $S$ is a two-dimensional Cohen-Macaulay local ring with minimal multiplicity and $(R,\mm)$ is a one-dimensional Cohen-Macaulay local ring which is a quotient of $S$, then every maximal Cohen-Macaulay $R$-module $M$ has non-decreasing Hilbert function. Here the $h$-th value of the Hilbert function of $M$ is defined as the length of $M\mm^h  / M\mm^{h+1} $. Therefore, Proposition \ref{HilbFun} implies that, under the above hypothesis, $\RR$ has non-decreasing Hilbert function for any maximal Cohen-Macaulay ideal $I$; in particular, if $I$ is a canonical ideal, $\RR$ is a one-dimensional Gorenstein local ring with non-decreasing Hilbert function for all $a,b \in R$. See \cite{P} for several explicit cases in which the above hypothesis hold. } 
\end{remark} 

\section{\bf  Construction of almost symmetric semigroups}
In order  to obtain Gorenstein  local rings with decreasing Hilbert function, by Corollary \ref{almsym} it is enough to find   almost Gorenstein rings $R$ such that $H_R(h-2)>H_R(h)$; in this   section we construct  infinitely many semigroup rings  verifying these  conditions. \\
First, we briefly recall some definitions and properties about numerical semigroup theory that we need.
A {\em numerical semigroup} $S$ is a submonoid of the natural numbers such that $|\NN \setminus S| < \infty$. The maximum element of $\NN \setminus S$ is called {\em Frobenius number} of $S$ and it will be denoted by $f(S)$. 
If $S$ is generated by $n_0  \leq n_1 \leq \dots \leq n_{\nu-1}$, we write $S= \langle n_0, \dots, n_{\nu-1} \rangle$. It is well-known that a numerical semigroup has an unique minimal system of generators and its cardinality is the {\em embedding dimension} $\nu$  of $S$. The smallest non-zero element of $S$ is $n_0 $; it is called {\em multiplicity} of $S$ and we will denote it by $e(S)$ or simply $e$, if the semigroup is clear from the context.
A {\em numerical semigroup ring} is a local ring of the form $k[[S]]:=k[[t^{n_0},t^{n_1}, \dots,t^{n_{\nu-1}}]]$, where $S= \langle n_0,n_1 \dots, n_{\nu-1} \rangle$ is a numerical semigroup and $k$ a field.
A {\em relative ideal} $E$ of a numerical semigroup is a subset of $\mathbb{Z}$ such that there exists $x\in \NN$ for which $x+E \subseteq S$ and $E+S \subseteq E$; if $E$ is contained in $S$ we say that $E$ is a (proper) ideal of $S$. 
An example of ideal is the {\em maximal ideal} $M=M(S):=S \setminus \{0\}=v(\mm)$, where $\mm=(t^{n_0},t^{n_1}, \dots,t^{n_{\nu-1}})$ and ${ v} :k((t))\longrightarrow {\mathbb Z}\cup\{\infty\}$ is the usual valuation. An example of a relative ideal is the {\em standard canonical ideal} $K(S):=\{x \in \NN \ | \ f(S)-x \notin S \}$; more generally, we call {\em canonical ideals} all the relative ideals $K(S)+z$ for any $z \in \ZZ$. \\ The properties of a semigroup ring  are strictly related to those of the associated numerical semigroup. In particular: 
  $H_{K[[S]]}(h)=|v(\mm^h) \setminus v(\mm^{h+1})|=|hM  \setminus (h+1)M |$ for each $h\geq1$. {
  We shall denote this function  by $H_S$ and its values by$\, [1,\nu(S),H_S(2),\dots,e,\rightarrow ]$.} It is well-known that:\\
$\bullet$ $k[[S]]$ is Gorenstein if and only if $ K(S)=S$; in this case $S$ is said to be {\em symmetric}; \\
$\bullet$ $k[[S]]$ is almost Gorenstein if and only if $M+K(S)=M$; in this case $S$ is said to be {\em almost symmetric}.

\begin{definition} 
Let $S$ be a numerical semigroup. 
\enui
	\item If $s$ is an element of $S$, the {\it order} of $s$ is  $ord(s):=\max \{i \ | \ s\in iM \}$.
	\item The {\em Ap\'ery set} of $S$ is $ Ap(S) := \{s \in S \ | \ s-e \notin S \}$, shortly $Ap$; it has cardinality $e$.
	\item $Ap_k:=\{ s\in  Ap  \ | \ ord(s)=k \}$.
	\item $D_h:= \{s \in S \ | \ ord(s)=h-1 \ {\rm and} \ ord(s+e)>h \}$, $D_h^t:=\{s \in D_h \ | \ ord(s+e)=t \}$.
	\item $C_k:=\{s \in S \ | \ ord(s)=k \ {\rm and} \ s-e \notin (k-1)M \}$. 
\end{enumerate}
\end{definition}
\noindent We notice that $C_k=  Ap_k \, \bigcup \, \{\cup_h \,\, (D_h^k + e) \ | \ 2 \leq h \leq k-1 \}$ and $H_S(k-1)-H_S(k)=|D_k|-|C_k|$, see for instance \cite{CJZ} and \cite{PT}.\\[2mm]
The following theorem of H. Nari characterizes the almost symmetric numerical semigroups by means of their Ap\'ery sets. First we recall that a {\em pseudo-Frobenius number} of $S$ is an integer $x \in \mathbb{Z} \setminus S$ such that $x + s \in S$ for any $s \in M $. We denote the set of pseudo-Frobenius number of $S$ by $\PF(S)$; it is straightforward to see that $f(S) \in \PF(S)$. The cardinality of $\PF(S)$ is the {\em type} of $S$ and it will be denoted by $t(S)$;   it is well-known that $t(S)=t(k[[S]])$.  

\begin{theorem} {\rm\cite[Theorem 2.4]{Nari}}\label{nari}
Let $S$ be a numerical semigroup. Set $ Ap =A\cup B$, where $A:=\{0 <  \alpha_1 < \dots < \alpha_m\}, \, B:= \{\beta_1 < \dots  < \beta_{t(S)-1} \}$, with $m=e \!-\!t(S)$, and $\PF(S)=\{\beta_i \!-\!e   \ | \ 1 \leq i \leq t(S)\!-\!1 \}\cup \{\alpha_m \!-\! e =f(S)\}$. 
The following conditions are equivalent:
\begin{enumerate}
	\item[\it i.] $S$ is almost symmetric; 
	\item[\it ii.] $\alpha_i + \alpha_{m-i}= \alpha_m$ for all $i \in \{1,2, \dots, m-1 \}$ and $\beta_j \!+\! \beta_{t(S)-j}=\alpha_m \!+\! e $ for all $j \in \{1,2, \dots, t(S)-1 \}$. 
\end{enumerate}
\end{theorem}

Since we are looking for a semigroup with decreasing Hilbert function, we need that $| Ap_2 | \geq 3$, by \cite[Corollary 3.11]{DDM};\, then we focus on the simpler case, $| Ap_2 |=3$.

\begin{proposition}\label{ap23}
Assume $|Ap_2|=3, \ Ap_k=\emptyset$ for all $k\geq 3$ and $H_S$ decreasing. Then $S$ cannot be almost symmetric.

\end{proposition}
\begin{proof} The assumptions on the Ap\'ery set imply that $\nu =e-3$. Since $H_S$ decreases, by \cite[Theorem 4.2.3]{OT} there exist $n_1\neq n_2\in Ap_1$ such that  $Ap_2=\{2n_1, n_1+n_2,2n_2 \}$ and we assume $n_1<n_2$. The element $n_1-e$ is not a pseudo-Frobenius number, because $2n_1-e \notin S$; therefore if $S$ is almost symmetric, with the notation of the previous theorem, $A$ is non-empty. It follows that $ord(\alpha_m)>1$ and so $\alpha_m=2n_2$. On the other hand by \cite[Proposition 4.3.1]{OT} we have $3n_2-e\in Ap$, that is a contradiction because $3n_2-e>2n_2.$
\end{proof}

According to the above proposition, we consider the next case $|Ap_3|=1$. In this context the following proposition holds:

\begin{proposition}\label{to} {\rm\cite[Proposition 3.4]{OT}}
Assume   $| Ap_2 |=3,\, | Ap_3 |=1$   and $H_S$   decreasing. Let $\ell = \min \{h \ | \ H_S \ {\rm decreases \ at \ level} \ h \}$ and let $d  = \max \{ord(\sigma) \ | \ \sigma \in Ap  \}$. Then $\ell \leq d \, $ and there exist $n_1, n_2 \in  Ap_1$ such that for $2\leq h\leq \e$ 
$$C_h=\{hn_1,(h-1)n_1+n_2,\dots,n_1+(h-1)n_2,hn_2\}=Ap_h\cup(D_{h-1}+e)$$
$$D_{\ell }+e=\{(d  +1)n_1, \ell   n_1+n_2,(\ell   -1)n_1+2n_2, \dots,(\ell   +1)n_2 \}.$$

Further, if $(\ell   ,d  )\neq (3,3), $ then $Ap_k={kn_1}$, for \,$3 \leq k \leq d  $. 
\end{proposition}      

The next proposition shows that, in the setting of the previous one, we only need to find an almost symmetric semigroup with decreasing Hilbert function. This is not true in general, for instance the numerical semigroup 
$S=\langle 30, 35, 42, 47, 108, 110, 113, 118, 122, 127, 134, 139 \rangle$ is almost symmetric and its Hilbert function is $H_S=[1, 12, 17, 16, 25, 30 \rightarrow]$. Therefore $H_S$ decreases, but $H_S(h-2) \leq H_S(h)$ for any $h \geq 2$.

 \begin{proposition}\label{to1}
Assume that $|Ap_2|=3, |Ap_3|=1$, and $H_S$ is decreasing. Let $\e$ be the minimum level in which the Hilbert function of $S$ decreases:
\enui
\item If $\e\geq 3$, then
   $H_S(h)=H_S(\ell   -1)$ for all $h \in [1,\ell -1]$. Further,  $H_S(\ell   -2)-H_S(\ell   )=1$.
\item If $S$ is almost symmetric then   $\alpha_m=d n_1$ and $kn_1 \in A$ for any $k\leq d$ (see Theorem \ref{nari} and Proposition \ref{to} for the notation).
\item If $S$ is almost symmetric, then  $\ell\geq 3$.  
\denu
 \end{proposition}

\begin{proof}
  {\it i}. By Proposition \ref{to}, if $2<h<\ell$ we have $C_{h}= (D_{h-1}+e) \cup Ap_{h}$ and then $|D_{h-1}|=|C_{h}|-1$, since $|Ap_h|=1$. Moreover, in this range, we have $|C_{h}|=|C_{h-1}|+1$, because of the previous proposition. Hence for any $h=2, \dots, \ell   -1$ we have 
$$
H_S (h-1) - H_S (h)=|D_h|-|C_h|=|C_{h+1}|-1-(|C_{h+1}|-1)=0. 
$$ 
Consequently we get $H_S(1)=H_S(2)= \dots = H_S(\ell   -1)$. As for the last part of the statement it is enough to note that, by the previous proposition, we have 
$$H_S(\ell-1)-H_S(\ell)=|D_{\ell}|-|C_{\ell }|=\ell+2-(\ell+1)=1.$$ 
\noindent {\it ii}. Of course $\alpha_m$ is the greatest element of the Ap\'ery set and in our case it can be either $dn_1$ or $2n_2$. If $\alpha_m=2n_2$, then there would exist $n \in Ap$, such that  $(d -1)n_1+n=2n_2$, but it is impossible, since $d \geq 3$ implies $ord(2n_2)\geq 3$.  Clearly  $k n_1\in A$, because $k n_1+(d-k)n_1=dn_1$.   \\
{\it iii}.  Assume $\ell= 2$. By Proposition \ref{to}, $(d +1)n_1-e\in D_2$  and so  $ord ((d +1)n_1-e)=1$;  
consequently, if $S$ is almost symmetric, there exists $n \in Ap$ such that   $(d +1)n_1-e+n=d  n_1 +ke,$ with either $k=0$ or $k=1$. Hence $n_1+n= (1+k)e\leq 2e$, impossible.  
\end{proof}

The next construction is based on Proposition \ref{to} and indeed we will prove that it defines almost symmetric numerical semigroups with decreasing Hilbert functions satisfying the assumptions of Proposition \ref{to}.

\begin{construction}\label{asd}
{\it Let $\ell\in \NN, \ell\geq 4,\, \ell   \notin\{ 14+22k, \, 35+46k \ | \ k\in\NN\},$   let  
\\[1mm] \centerline{$e:= \ell^2+3\ell  +4 $}\\[1mm]
 $\left[\begin{array}{llllllll}n_1:=\ell  ^2+5\ell  +3&=&e+(2\ell  -1),  & n_2:= 2\ell  ^2+3\ell  -2&=&e+(\ell  ^2-6) , 
& {\it if}&  \ell  \,\, {\it  odd} \\[1mm]
 n_1:=\ell  ^2+4\ell  +1&=&e+( \ell  -3) , & n_2:= 2\ell  ^2+2\ell  -2&=&e+(\ell  ^2-\ell  -6) , & {\it if}& \ell  \,\, {\it even}\
\end{array}\right.$\\[2mm]
and let $S$ be the semigroup generated by the subset $\Gamma\subseteq \NN$   \\[1mm]
\centerline{  $\Gamma= \{e,n_1 ,n_2\}\,\cup\{t_1,t_2\}\cup   \{s_{p,q}  \}\!\cup\! \left\{r_{p,q}\right\}\setminus\{n_1\!+\!n_2, 2n_2\} $}
 where:   \\[2mm]
\centerline{$  
 \begin{array}{cllr}
  \{s_{p,q} \}\,\, &= 
  \left\{\,pn_1\!+\! qn_2\!-\! (p\!+\! q\!-\!  2)e \ | \  0\leq p\leq \e, \, \,1\leq q\leq \e+1, \,  \,\,  2\leq p\!+\! q\leq \ell  \!+\! 1 \right\}
\\[2mm]
\{r_{p,q}\}\,\,&=\{ \, \ell  n_1 \!+\!   e - s_{p,q} \ | \ 2\leq p \!+\!   q\leq \ell  \!+\!  1,\,\,p\geq 1,\,\,q\geq 1\,\}
 \\[2mm] 
  t_1   &=    (\ell  \!+\! 1)n_1 \!-\! (\ell  \!-\! 1)  e, 
  \\[2mm]
  t_2  &=\e n_1+e-t_1=(\e-1)e-n_1
  \\[2mm]
|\{s_{p,q} \}|&= (\ell^2  +3\ell ) /2:=u,  \quad |\{r_{p,q} \}|= (\ell^2  + \ell ) /2,  \quad |\Gamma|=|\{s_{p,q} \}|+|\{r_{p,q}\}|+3=e\!-\!\e\!-\!1    \\[2mm]
  (\ell  -1)e&=- \ell   n_1 + (\ell  +2)n_2  \end{array}  $  }  }
 \end{construction}
We note that the elements $\{e, n_1, n_2, t_1\}\cup\{s_{p,q} \}$ are given to obtain  the structure of  $S$ required in Proposition {\ref{to}}, with $d=\ell$, while the elements $\{r_{p,q} \}\cup\{t_2\}$, impose the almost symmetry of $S$ following Theorem \ref{nari}   (with $\alpha_m=\e n_1$, and $B\supseteq \{t_1\} \cup \{s_{p,q} \}\setminus \{s_{0,\e+1}\}$). \ 
In this construction, for $s \in\{qn_2-(q-2)e, 2\leq q\leq \ell  +1\} \cup\{n_2\}\cup\{  p n_1, 1\leq p\leq \ell  \}$, we don't need to add the corresponding $\e n_1\!+\!e\!-\!s $, or $\e n_1-s$, because such element   is already inside $S$ (see Lemma \ref{costr}.{\it\,i}).
\begin{remark}  
 {\rm  Looking  for  an almost symmetric semigroup  $S$ satisfying   Proposition \ref{to}, with $d=\e$, it is natural to impose $e\geq\ell  ^2+3\ell  +4$.
 In fact the first idea to construct this   semigroup  is to impose that a system of generators of $S$   contains the set \\
 \centerline{$ \{e,n_1 ,n_2\} \cup \{\e n_1-n_2\}\cup\{t_1,t_2\}\cup \{s_{p,q}  \} \cup \big\{\alpha_m+e-s , s \in \{s_{p,q} \}\big\} \setminus\{n_1\!+\!n_2, 2n_2\} $.}
   By counting the number of conditions  for a given $\e$,   we get \\
   \centerline{$e(S)\geq {4+2\displaystyle{\frac{  \ell^2+3 \ell   }{2}}} =\ell  ^2+3\ell  +4$.} 
 Further, following Proposition \ref{to}, we need \\
 \centerline{ $Ap= Ap_1\cup \{ 2n_1,n_1+n_2,2n_2\}\cup\{ kn_1,3\leq k\leq \ell  \}$ } thus the embedding dimension of $S$ must be $\nu=e(S)-\ell  -1$. If $|\{r_{p,q}  \}| = |\{s_{p,q} \}|-\e$ and $\e n_1-n_2\in\{s_{p,q}\}$, we could fix the minimal value  $e=\ell  ^2+3\ell  +4,$ for the multiplicity. This happens if we define $e, n_1$, and $n_2$   as in Construction \ref{asd}.  In fact  this choice gives    the basic relation  \centerline{\quad$\ell   n_1=(\ell  +2)n_2 -(\ell  -1)e$ }
 which  assures that,   for $ \,2\leq q \leq \ell   ,\,  \e n_1+e-s_{0,q}\in \{s_{p,q}\}$  and so it reduces the number of independent conditions to  $2|\{s_{p,q} \}| -\e +3$  \,\,(see Lemma \ref{costr}).} 
\end{remark}  
 We give some examples before to prove of the exactness of the construction.
\begin{example}\label{ex4-7} In this example we show the almost symmetric semigroups constructed by means of the above algorithm for $\e=4,5 $ and the Hilbert function $H_S$ of $R=k[[S]]$.
\begin{enumerate}
\item[$\e\!=\!4.$] The semigroup $S$ is minimally generated by 
$\{ 32$,  33$=n_1$,  38$=n_2$, 69, 72, 73, 74, 75, 77, 78, 79, 
  80, 81, 82, 83, 84, 85,  86, 87, 88, 89, 90, 91, 92, 93, 94, $95\}.$ Moreover  {$Ap_2=\{66,71,76\},\,\, Ap_3=\{99=3n_1\},\,\, Ap_4=\{132=\e n_1\},$}  
$\PF(S)=\{ $37, 39, 40, 41, 42, 43, 44, 45, 46, 47, 48, 49, 50, 
  51, 52, 53, 54, 55, 56, 57, 58,59, 60, 61,63, $100  \} $.  \\ \centerline{$H_S=[ 1,27, 27, 27, 26, 27, 29, 30, 31, 32 \rightarrow]$}.

\item[$\e\!=\!5.$] The minimal generating system of $S$ is $\{ $44, 53$=n_1$, 63$=n_2$, 117, 125, 127, 134, 135, 136, 137, 
  142, 143, 144, 145,  146, 147, 152, 153, 154, 155, 156, 157, 162, 163, 164, 
  165, 166, 167, 172, 173, 174, 175, 182, 183, 184, 192, 193, 202$\}$. \\{$Ap_2=\{106,116,126\},\,\, Ap_3=\{159\},\,\, Ap_4=\{212\},\,\,Ap_5=\{265\}$} 
$\PF(S) = \{ $72, 73, 81, 82, 83, 90, 91, 92, 93, 98, 99, 100, 101, 
  102, 103,  108, 109, 110, 111, 112, 113, 118, 119, 120, 121, 122, 123, 128, 129, 130, 131, 138, 139, 140, 148, 149, 221 $\}$. \\   \centerline{$H_S=[ 1, 38, 38, 38, 38, 37, 44 \rightarrow ]$}
\denu
\end{example} 
\noindent To validate Construction \ref{asd}, we need  some technical lemmas proved in the appendix.

\begin{theorem}\label{l1} With the assumptions of Construction \ref{asd}:  \enui
\item  The ring $R=k[[S]]$ is almost Gorenstein with Hilbert function decreasing at level $\e$: 
$$H_R=[1,\nu,\nu,\dots,\nu, \nu-1, H_R(\ell+1) \dots].$$

\item The embedding dimension of $R$ is $\,\,\nu(R)= e-(\ell  +1)=\e^2+2\e+3$. 
 \item The Cohen-Macaulay type of $R$ is $\,\,t(R)=\nu(R)-1=\e^2+2\e+2$.
\denu
\end{theorem} 
\begin{proof} 
{\it i}. By Proposition \ref{apery}, we know the Ap\'ery set of $S$ and the subsets $A,B$ of Theorem \ref{nari}. Then $S$ is almost symmetric, since the elements of $A $ and $B$ verify the conditions of Theorem \ref{nari}, respectively by Lemma \ref{costr}.{\it i} \, and by the  definition  of $\{r_{p,q}\}$.\\
Now we show that the Hilbert function of $k[[S]]$ decreases at level $\e$: we shall prove that \\
$ \bullet$  for each $k\in[2,\e+1]$, $ C_{k }=\{ k    n_1 , (k   -1)n_1+n_2 , (k   -2)n_1+2n_2,  \dots,k   n_2 \},\,\, |C_{k }|=k+1 $ 
\\ 
$ \bullet$ for each $k\in[2,\e-1],$
$D_{k}+e=\{k   n_1+n_2,(k   -1)n_1+2n_2, \dots,(k  +1)n_2 \}=C_{k+1}\setminus \{(k+1)n_1\},$
\par $|D_k|=k+1 $
\\
$ \bullet$ $D_{\ell }+e=\{(\e  +1)n_1, \ell   n_1+n_2,(\ell   -1)n_1+2n_2, \dots,(\ell   +1)n_2 \},\,\, |D_{\ell }|=\e+2.$
 \\[2mm]
 Hence the thesis follows, recalling that for each $k\geq 2: \,\,H(k)=H(k-1)+|C_{k }|- |D_{k }|$.\\[2mm] 
 Let $s=an_1+bn_2, $ \, with $   \,a+b=k\in[3,\e+1]$  and $a<k$, if $k\neq \e+1$. First we prove that \\
 \centerline{ if\,\, $an_1+bn_2= d +e,$ with $d\in D_h$, then $ord(an_1+bn_2)\!=\!h+1 \quad (1)$}  \\ In fact, if  $\ ord(an_1+bn_2)=k'=h+p,$ with $p\geq 2$, we know, by \cite[Proposition 2.2.1]{OT}, that given a maximal expression $an_1+bn_2=\sum a_i n_i ,\,$ with $n_i\in Ap_1,\,\sum a_i=k', \,$ for any  $y=\sum b_i n_i , \,  0\leq b_i\leq a_i,\, \sum b_i=q\leq p+1$, then $y\in Ap_q$. Since $k'>3$ and $Ap_3=\{3n_1\}$, this would imply $an_1+bn_2=k'n_1$, impossible by Lemma \ref{costr1}.{\it ii}. Hence $an_1+b n_2\in C_{h+1}\cap (D_{h}+e)$. 
 
By definition, \,$C_2=Ap_2$  and  by   Proposition \ref{apery}, $D_2 \supseteq \{2n_1+n_2-e, n_1+2n_2-e, 3n_2-e \}$. Then $D_2=\{2n_1+n_2-e, n_1+2n_2-e, 3n_2-e \}$, otherwise the Hilbert function decreases at level 2, impossible by Proposition \ref{to1}.{\it iii}. Hence $C_3=(D_2+e)\cup \{3n_1\}$. 

Now we proceed by induction on $k$. First we recall that, if $x\in C_k$  has maximal representation $x=\sum a_in_i, \sum a_i=k, n_i\in Ap_1$ and $y=\sum b_in_i, 0\leq b_i\leq a_i$ with $\sum b_i=h$, then $y\in C_h$ by \cite[Proposition 1.4.1]{OT}. In our case $C_2=\{2n_1,n_1+n_2,2n_2\}$, hence $ C_k\subseteq \{an_1+bn_2 \ | \ a+b=k\}$.
Assume $3\leq k\leq \e $ and the thesis true for $k-1$.
Therefore we know the structures of $C_2, ... C_k$,   $D_2,...,D_{k-1}$. Let  $an_1+bn_2$ with $  \,a+b=k+1\in[4,\e+1]$  and $a<k+1$, if $k\neq \e$. 
Then, by Lemma \ref{asd},\,\, $s_{a,b}=an_1+bn_2-(k-1)e\in Ap_1$ is such that $ord( s_{a,b}+(k-1)e)\geq k+1$. Moreover, since $s_{a,b}+e \notin D_2$, we know that its order is 2; hence there exists $r\in[1,k-2]$ such that $ord ( s_{a,b}+re) =r+1$ and 
$ord(s_{a,b}+(r+1)e) >r+2$ i.e. $s_{a,b}+re\in D_{r+2},$ with $r+2\leq k $. If $r+2<k$, by induction there would exist  $ a'n_1+b'n_2-e\in  D_{r+2}$  such that $s_{a,b}+re=a'n_1+b'n_2-e=s_{a',b'}+(a'+b'-3)e$: impossible because $s_{a,b}$ and $s_{a',b'}$ have distinct residues $(mod\, e),$ by Lemma \ref{costr1}.{\it ii}. Hence $r=k-2,\,\,  an_1+bn_2-e\in D_k $ and $an_1+bn_2\in C_{k+1}$ by $(1)$. This proves {\it i\,}.
\\[2mm] 
   {\it ii}. Since there are $\ell+1$ elements of the Ap\'ery set with order greater than $1$, we have $\nu(R)=|Ap_1|+1=e-1-(\e+1)+1=\e^2+2\e+3$. 
 \\[2mm]
{\it iii}.   The Cohen-Macaulay type of $k[[S]]$ is the cardinality of the Pseudo-Frobenius set of $S$: \\
 $t(R)=|B|+1=\displaystyle{\frac{  \ell^2+3 \ell}{2} -1+ \frac{\ell^2+  \ell}{2} +3}=\e^2+2\e+2$.    \end{proof}

Without using Construction \ref{asd}, but by similar techniques, it is possible to construct other almost symmetric semigroups such that    $H_S(h-1)>H_S(h)$, even if $h=2,3$,   as shown in the first two semigroups of the next example. Moreover the last one is another almost symmetric semigroup with $|Ap_2|=3, |Ap_3|=1$  and decreasing Hilbert function.

\begin{example} \label{AlmostSymmetric}
\enui
\item The numerical semigroup $S=\langle 33, 41, 42, 46, 86, 90, 91, 95, 96, 97, 98,$  $100, 101, 103, 104, 105, 106, 109, 110, 111, 113, 114, 118, 122 \rangle$ is almost symmetric and its Hilbert function $[ 1, 24, 23, 23, 31, 33 \rightarrow ]$ decreases at level 2. Moreover \\ 
\centerline{$Ap_2=\{82, 83, 84, 87, 88, 92, 127 \},\,\, Ap_3=\{126 \},\,\, Ap_4=\{168 \},\,\, Ap_k=\emptyset$ \ if $k \geq 5$. \ \ \ \ \ }

\item The numerical semigroup $S=\langle 32, 33, 38, 58, 59, 60, 61, 62, 63, 67,68,69,72, 73, 74,75,$ $77, 78, 79, 80, 81, 82, 83, 84, 85, 86, 87, 88 \rangle$ is almost symmetric with   Hilbert function  \,\,
$[ 1, 28, 28, 27, 27, 29, 30, 31, 32 \rightarrow ]$ decreasing at level 3 and \\ 
\centerline{$Ap_2=\{ 66, 71, 76, 121\},\,\, Ap_k=\emptyset$ \ if $k \geq 3$. \ \ \ \ \ }
 \item The numerical semigroup $S=\langle 30,33,37,64,68,71,73,75,76,77, 78,79,80,81,82, 83, $ $84,85,86,87,88,89,91,92,94,95,98,101 \rangle$ is almost symmetric with   Hilbert function \,\,$[ 1, 25, 25, 25, 24, 27, 28, 29, 30 \rightarrow ]$  \\   
\centerline{$Ap_2=\{66, 70, 74\},\,\, Ap_3=\{99\},\,\, Ap_4(S)=\{132\},\,\,Ap_k=\emptyset$ \ if $k \geq 5$.\ \ \ \ \ } 
\denu
\end{example}

\section { \bf The Gorenstein case}
 In this section we give explicit examples of local one-dimensional Gorenstein rings with decreasing Hilbert function and other interesting examples. Several computations are performed by using the GAP system \cite{GAP} and, in particular, the NumericalSgps package \cite{DGM}.\\[2mm]
Luckily, if $R$ is a numerical semigroup ring and $b=t^m \in R$, with $m$ odd, then $R(I)_{0,-b}$ is a numerical semigroup ring and it is exactly the ring associated with the so-called numerical duplication. Anyway we note that in general, for other choices of $a$ and $b$, the ring $\RR$ is not a numerical semigroup. For example if $a=-1$ and $b=0$ it is isomorphic to the amalgamated duplication that, in this case, is reduced but not a domain; while $R(I)_{0,0}$ is isomorphic to the idealization and then it is not reduced. In this section we describe the particular case of the numerical duplication, that is probably the easiest case; we show the most notable and simple examples among the various we have constructed. 

Let $S$ be a numerical semigroup, $b \in S$ be an odd integer and $E$ be a proper ideal of $S$. The {\em numerical duplication} of $S$ with respect to $E$ and $b$, introduced in \cite{DS}, is the numerical semigroup
$$
S \du E := \{2 \cdot S \} \cup \{2 \cdot E +b\}, 
$$
where $2 \cdot X=\{2x \ | \ x \in X\}$ for any set $X$; we note that $2 \cdot X$ is different from $2X=X+X$. 

As mentioned above, if $R=k[[S]]$ is a numerical semigroup ring and $b=t^m \in R$ with $m$ odd, it is proved in \cite[Theorem 3.4]{BDS} that\\[2mm]
\centerline{$R(I)_{0,-b}$ is isomorphic to $k[[S \! \Join^m \! E]]$,}\\[2mm] 
where $E:=v(I)$ is the valuation of $I$, see \cite{BDS} for more detail. We recall that $I$ is a canonical ideal of $R$ if and only if $v(I)$ is a proper canonical ideal of $S$; hence $S \du E$ is symmetric if and only if $E$ is a canonical ideal, see also \cite[Proposition 3.1]{DS} for a simpler proof.

It is easy to compute the generators of $S \du E$, in fact if ${\rm G}(S)=\{ n_1, \dots, n_r \}$ is the set of the minimal generators of $S$ and $E$ is generated, as ideal, by $\{m_1, \dots, m_s \}$, then\\[2mm]
\centerline{$S\du E= \langle 2n_1, \dots, 2n_r, 2m_1+b, \dots, 2m_s+b \rangle$.}\\[2mm]
In particular, we recall that $K(S)$ is minimally generated by the elements $f(S)-x$, where $x \in {\rm PF}(S)$, and therefore, if $E=K(S)+z$, the semigroup $S \du E$ is generated by \\[2mm]
\centerline{$\{2n_i,\,\,2(f(S)-x_j+z)+b \,\, | \,\, n_i \in {\rm G}(S), \,x_j \in {\rm PF}(S)\}$;}\\[2mm]
moreover if $S$ is almost symmetric, it follows from Theorem \ref{nari} that $S \du (K(S)+z)$ is minimally generated by $\{2n_i, 2z+b, 2x_j+2z+b \ | \ n_i \in {\rm G}(S), x_j \in {\rm PF}(S) \setminus \{f(S)\} \}$. Finally we remember that if $S=\langle s_1, \dots, s_\nu \rangle$ is a symmetric numerical semigroup then $k[[S]]:=k[[t^{s_1}, \dots,t^{s_\nu}]]$ is a one-dimensional Gorenstein local ring for any field $k$.\\[2mm]

\indent  For each $h\geq 4$, $h \notin\{ 14+22k, \, 35+46k \ | \ k\in\NN\}$, Construction \ref{asd} allows to produce Gorenstein rings whose Hilbert function decreases at level $h$, 
  while for $h=3$ we can use Example \ref{AlmostSymmetric}.{\it ii}.    The next example is useful to complete the case $h=2$.

\begin{example} \label{h=2}
Consider the numerical semigroup 
\begin{equation*}
\begin{split}
S=\langle &68, 72, 78, 82, 107, 111, 117, 121, 158, 162, 166, 168, 170, 172, 174, 176, 178, 180, 182, 184, \\
&186, 188, 190, 192, 194, 196, 197, 198, 200, 201, 202, 205, 206, 207, 209, 210, 211, 213, 215, \\
&217, 219, 221, 223, 225, 227, 229, 231, 233, 235, 237, 239, 241, 245, 249 \rangle.
\end{split}
\end{equation*}
 It is almost symmetric, has type $53$, and its Hilbert function is $[ 1, 54, 52, 50, 54, 64, 68 \rightarrow ]$. Consequently by Proposition \ref{Formula}, $k[[S\du K]]$ is a Gorenstein ring and has Hilbert function $[ 1, 107, 106, 102, 104, 118, 132, 136 \rightarrow ]$ decreasing at level 2 for any canonical ideal $K$ and for any odd $b \in S$. 
\end{example} 

The next lemma will allow us to show that in general, even if $R$ is Gorenstein, there are no bounds for $H_R(h-1)-H_R(h)$. 

\begin{lemma} \label{maximal} Let $R^{(0)}$  m be a local ring. The following hold: 
\enui
\item Consider the ring $R^{(i+1)}:=R^{(i)}(\mm^{(i)})_{a^{(i)},b^{(i)}}$, where $\mm^{(i)}$ is the maximal ideal of $R^{(i)}$ and $a^{(i)},b^{(i)}$ are two elements of $R^{(i)}$. Then
$$H_{R^{(i)}}(h)=2 H_{R^{(i-1)}}(h)= \dots = 2^i H_{R^{(0)}}(h) \,\,{\text for \ any \ } h>0.$$
\item  If $R^{(0)}$ is almost Gorenstein and has  Cohen-Macaulay  type $t$, then $R^{(i)}$ is almost Gorenstein and, if $R^{0}$ is not a DVR, it has Cohen-Macaulay type $\, 2^i t + 2^i-1$.
\denu
\end{lemma}
 
\begin{proof}
The first point is a straightforward application of Proposition \ref{HilbFun}, while the second one follows from \cite[Proposition 2.9]{BDS2}, which says that $R$ is almost Gorenstein if and only if $R(\mm)_{a,b}$ is almost Gorenstein and in this case, if $R$ is not a DVR,   the Cohen-Macaulay type of $R(\mm)_{a,b}$ is $2t+1$.
\end{proof}

\begin{theorem} \label{corMax}  
For any integers $m \geq 1$ and $h>1, h \notin\{ 35+46k ,\, 14+22k \ | \ k\in\NN\}$, there exist infinitely many non-isomorphic one-dimensional Gorenstein local rings $R$ such that \,\,  $H_R(h-1)-H_R(h) > m$.  
\end{theorem}

\begin{proof}
If $h=2$ consider the ring $R^{(0)}=k[[S]]$ of Example \ref{h=2}, then  $R^{(i)}$ is almost Gorenstein and $H_{R^{(i)}}(1)=54\cdot 2^i,\,\, H_{R^{(i)}}(2)=52 \cdot 2^i, t(R^{(i)})=54\cdot 2^{i}-1$ \, by the previous lemma. We achieve the proof using  Corollary \ref{almsym} applied with  $R=R^{(i)}$: $H_{\RR}(1)-H_{\RR}(2)=t(R)- H_R(2)=54\cdot 2^i-1-52\cdot 2^{i}=  2^{i+1}-1>m,$ if $i\geq i_0=\lfloor \log_2 (m+1) \rfloor$.\\
If $h\geq 3$, consider an almost Gorenstein ring $R$ such that $H_R(h-2)-H_R(h)=n>0$, whose existence we proved in Theorem \ref{l1} for $h\geq 4$ and in Example \ref{AlmostSymmetric}.{\it ii} for $h=3$; then  apply the construction of the previous lemma with $i_0=\lfloor \log_2 (m/n) \rfloor + 1$.  With the notation of the previous lemma, it follows that $H_{R^{(i)}}(h-2)-H_{R^{(i)}}(h)=2^i n > m$ for any $i \geq i_ 0$. \\
Now, for each $h$ of the statement,  if $i\geq i_0$, $\omega$ is a canonical ideal of $R^{(i)}$ and $a,b \in R^{(i)}$, Corollary \ref{almsym} implies that the ring $R^{(i)}(\omega)_{a,b}$ has all the properties we are looking for. \\  
Clearly, for any $i \geq i_0$ we get infinitely many non-isomorphic rings, because their   Hilbert functions are different.
\end{proof}

From the proof it is clear that, for any such $m$ and $h$, in the rings of the previous theorem there are always non-reduced rings (idealization), reduced rings that are not integral domains (amalgamated duplication), and numerical semigroup rings (numerical duplication).

\begin{example}\label{es64} Consider the first numerical semigroup $S$ of Example \ref{ex4-7}. Set $b=33$ and $E=K(S)+101=K(S)+f(S)+1 \subseteq S$. Since we know the generators and the pseudo-Frobenius numbers of $S$, it follows from above that $k[[S \du E]]$ is equal to
\begin{equation*}
\begin{split}
k[[ &t^{64}, t^{66}, t^{76}, t^{138}, t^{144}, t^{146}, t^{148}, t^{150}, t^{154}, t^{156}, t^{158}, t^{160}, t^{162}, t^{164}, t^{166}, t^{168}, t^{170}, t^{172}, t^{174}, t^{176},\\ 
&t^{178}, t^{180}, t^{182}, t^{184}, t^{186}, t^{188}, t^{190}, t^{235}, t^{309}, t^{313}, t^{315}, t^{317}, t^{319}, t^{321}, t^{323}, t^{325}, t^{327}, t^{329}, t^{331}, t^{333},\\ 
&t^{335}, t^{337}, t^{339}, t^{341}, t^{343}, t^{345}, t^{347}, t^{349}, t^{351}, t^{353}, t^{355}, t^{357}, t^{361}]] 
\end{split} 
\end{equation*}
and is a one-dimensional Gorenstein local ring. Moreover Proposition \ref{Formula} implies that its Hilbert function is 
$[ 1, 53, 54, 54, 53, 53, 56, 59, 61, 63, 64 \rightarrow]$. 
\end{example}  

In the next example we show how the construction of Lemma \ref{maximal} and Theorem \ref{corMax} works.

\begin{example} Let $T^{(0)}$ be the second semigroup of Example \ref{ex4-7} and construct the numerical semigroups of Lemma \ref{maximal} applying the numerical duplication, that can be considered a particular case of the lemma.

\begin{itemize}
\item    $T^{(0)}$ is almost symmetric with type $37$ and
$$H_{T^{(0)}}=[1, 38, 38, 38, 38, 37, 44, \rightarrow ];$$ 
\item   $T^{(1)}:=T^{(0)} \!\Join\!^{53} M(T^{(0)})$ is almost symmetric with type $75$ and  
$$H_{T^{(1)}}=[1,76,76,76,76,74,88  \rightarrow ];$$ 
\item  $T^{(2)}:=T^{(1)} \!\Join\!^{141} M(T^{(1)})$ is almost symmetric with type $151$ and 
$$H_{T^{(2)}}=[1,152,152,152,152,148,176  \rightarrow ];$$ 
\item   $T^{(3)}:=T^{(2)} \!\Join\!^{317} M(T^{(2)})$ is almost symmetric with type $303$ and 
$$H_{T^{(3)}}=[1,304,304,304,304,296,352  \rightarrow ];$$ 
\item   $T^{(4)}:=T^{(3)} \!\Join\!^{669} M(T^{(3)})$ is almost symmetric with type $607$ and 
$$H_{T^{(4)}}=[1,608,608,608,608,592,704  \rightarrow ];$$ 
\item   $T:=T^{(4)} \!\Join\!^{1373} K$, where $K:=K(T^{(4)})+f(T^{(4)})+1 \subseteq T^{(4)}$, is symmetric and has Hilbert function 
$$H_{T}=[1,1215,1216,1216,1216,1200,1296,1408 \rightarrow ].$$ 
\end{itemize}
If we are looking for symmetric semigroups with bigger difference between $H(4)$ and $H(5)$, we can continue in this way   before to consider the numerical duplication with respect to a canonical ideal. Anyway, we note that in this example $T$ has $1215$ minimal generators included between $1408$ and $23835$.
\end{example}

\begin{example}
Consider the almost symmetric numerical semigroup 
$$T_0 \! = \! \langle 30,33,37,64,68,69,71,72,73,75,76,77,78,79,80,81,82,83,84,85,86,87,88,89,91,92  \rangle$$
that has Hilbert function $[ 1, 26, 26, 25, 24, 27, 28, 29, 30  \rightarrow ]$. 
Let $K'(T_i)$ be a proper canonical ideal of $T_i$ and $b_i$ an arbitrary odd element of $T_i$. All the following numerical semigroups are symmetric: 
\begin{itemize}
\item The semigroup $T_1:=T_0 \!\Join\!^{b_0} K'(T_0)$ has Hilbert function $$H_{T_1}=[ 1, 51, 52, 51, 49, 51, 55, 57, 59, 60 \rightarrow ];$$ 
\item The semigroup $T_2:=T_1 \!\Join\!^{b_1} K'(T_1)$ has Hilbert function $$H_{T_2}=[ 1, 52, 103, 103, 100, 100, 106, 112, 116, 119, 120 \rightarrow ];$$ 
\item The semigroup $T_3:=T_2 \!\Join\!^{b_2} K'(T_2)$ has Hilbert function $$H_{T_3}=[ 1, 53, 155, 206, 203, 200, 206, 218, 228, 235, 239, 240 \rightarrow ];$$ 
\item The semigroup $T_4:=T_3 \!\Join\!^{b_3} K'(T_3)$ has Hilbert function $$H_{T_4}=[ 1, 54, 208, 361, 409, 403, 406, 424, 446, 463, 474, 479, 480 \rightarrow ];$$ 
\item The semigroup $T_5:=T_4 \!\Join\!^{b_4} K'(T_4)$ has Hilbert function $$H_{T_5}=[ 1, 55, 262, 569, 770, 812, 809, 830, 870, 909, 937, 953, 959, 960  \rightarrow ].$$ 
\end{itemize}
  Since a symmetric numerical semigroup is almost symmetric the Hilbert functions above can be computed from the one of $T_0$ by means of Proposition \ref{Formula}. 
\end{example}

The next two examples show that it is possible to find symmetric semigroups with decreasing Hilbert function even if we start with {\em non-almost symmetric semigroups}.\par 
Further we recall that, by \cite[Corollary 4.11]{OT}, in a symmetric semigroup with decreasing Hilbert function the difference between the multiplicity and the embedding dimension has to be greater or equal to $5$: in the following example is $6$.

\begin{example} \label{e-v}
Consider 
$
S:= \langle 30,33,37,64,68,69,71,72,73,75, \, \longrightarrow \, 89,91,92,95 \rangle
$ that has Hilbert function $[1, 27, 26, 25, 24, 27, 28, 29, 30, \rightarrow]$  and set $K:=K(S)+66 \subseteq S$. Then the semigroup $S\!\Join\!^{33}K$ is symmetric and has Hilbert function $[ 1, 54, 55, 55, 54, 57, 58, 59, 60 \rightarrow ]$.  We also note that $S$ is not almost symmetric by Proposition \ref{Formula}. 
\end{example} 

It is possible to define the numerical duplication $S \du E$, even if  the ideal $E$ is {\em not contained} in $S$; in this case we have to require that $E+E+b \subseteq S$, that is true if $E \subseteq S$, otherwise the set $S \du E$ is not a numerical semigroup. In \cite[Corollary 3.10]{S1} it is proved that, even if $E$ is not proper, $S \du E$ is symmetric if and only if $E$ is a canonical ideal; actually every symmetric numerical semigroup can be constructed as $S \du K(S)$ for some $S$ and some odd $b \in S$ (see also \cite[Proposition 3.3]{S1} and \cite[Section 3]{S2}). However, if {\em the ideal is not proper}, the Hilbert function of the numerical duplication can be different from the expected one; on the other hand the next examples show that also in this case it is possible to find symmetric semigroups with decreasing Hilbert function.

\begin{example} \label{l<4}
In \cite[Example 3.7]{OT} it is showed the following numerical semigroup 
$$S \! = \! \langle 30, 33, 37, 73, 76, 77, 79, 80, 81, 82, 83, 84, 85, 86, 87, 88, 89, 91,92, 94, 95, 98, 101, 108 \rangle$$
that has Hilbert function $[1, 24, 25, 24, 23, 25, 27, 29, 30 \rightarrow ]$. 

Let $K:=K(S)$. The following semigroups are all symmetric and all of them, but $H_2$, have decreasing Hilbert function. It is easy to see that, for the following choices of $b$, one has $K+K+b \subseteq S$, but $S \du K$ cannot be realized as a numerical duplication with respect to a proper ideal.

\begin{itemize}
\item The semigroup $H_1:=S \!\Join\!^{79} K$ has Hilbert function $[1, 44, 41, 40, 52, 58, 60,   \rightarrow ]$;
\item The semigroup $H_2:=S \!\Join\!^{81} K$ has Hilbert function $[1, 43, 45, 47, 52, 54, 56, 58, 60 \rightarrow ]$; 
\item The semigroup $H_3:=S \!\Join\!^{85} K$ has Hilbert function $[1, 44, 42, 45, 52, 54, 58, 60  \rightarrow ]$;
\item The semigroup $H_4:=S \!\Join\!^{87} K$ has Hilbert function $[1, 46, 48, 47, 49, 51, 56, 58, 60 \rightarrow ]$;
\item The semigroup $H_5:=S \!\Join\!^{93} K$ has Hilbert function $[1, 47, 49, 48, 48, 50, 55, 58, 60  \rightarrow ]$.
\end{itemize}

Note that $H_5$ have the same Hilbert function of the numerical duplication with respect to a proper canonical ideal of $S$.

\end{example}

If one consider the semigroups constructed in the previous section and their numerical duplications with respect to non proper canonical ideals, it is possible to find symmetric semigroups whose Hilbert functions {\em decrease at more levels}. For instance the next example shows a symmetric semigroup that decreases $13$ times. We also note that it decreases at level $14$, thus this suggests that the restrictions of Theorem \ref{corMax} can be removed.

\begin{example}\label{l15}
Let $S$ be the semigroup of Construction \ref{asd} with $\ell=15$, that has $258$ minimal generators. According to GAP \cite{GAP}, the symmetric semigroup $S\!\Join\!^{957} K(S)$ has Hilbert function 
$$
[ 1, 514, 514, 513, 512, 511, 510, 509, 508, 507, 506, 505, 504, 503, 502, 500, 523, H_S(17), \dots ].
$$
\end{example}

In the last example we show that the numerical duplication of $S$ with respect to a canonical ideal can have decreasing Hilbert function, even if that of $S$ is non-decreasing. \par Further, among the symmetric semigroups with decreasing Hilbert function this is the semigroup with {\it the smallest multiplicity and embedding dimension} that we know.

\begin{example} \label{mol}
Consider  
$S=\langle 19, 21,$ 24, 47, 49, 50, 51, 52, 53, 54, 55, 56, 58, $60 \rangle,$ 
shown in \cite[Example 3.2.1]{OT}, and let $T=S \!\Join\!^{49} K(S)$. The Gorenstein local ring $k[[T]]$ is
\begin{equation*}
\begin{split}
k[[&t^{38}, t^{42}, t^{48}, t^{49}, t^{94}, t^{100}, t^{101}, t^{102}, t^{104}, t^{105}, t^{106}, t^{107}, t^{108}, t^{109}, t^{110}, t^{111}, t^{112}, t^{113},  
t^{115}, t^{116}, \\
&t^{117}, t^{119}, t^{120}, t^{121}, t^{123}, t^{127}]].
\end{split}
\end{equation*}
Even if $k[[S]]$ has non-decreasing Hilbert function  $[1, 14, 14, 14, 16, 18, 19 \rightarrow ]$, the Hilbert function of $k[[T]]$ is $[ 1, 26, 25, 25, 32, 38 \rightarrow ]$; we also note that its multiplicity is $38$. 
\end{example} 

\section { \bf Appendix}
  In this appendix  we illustrate the technical lemmas necessary to prove Theorem \ref{l1}; in the sequel we shall assume  $e,\,n_1,\,n_2,\,\e, \, S$ \, be as defined in Construction \ref{asd}.
\begin{lemma} \label{copr} We have:
\enu
  \item If $\ell   $ is odd then  \ $GCD(e,n_1,n_2)=1\II     \ell  \notin\{35+46k, k\in \NN\}$.
    \item   If $\ell$ is even then \ $GCD(e,n_1,n_2)=1\II \ell  \notin\{ 14+22k,\,\, k\in \NN\}$. 
\denu
\end{lemma}
\begin{proof}{\it i}.   $\left\{ \begin{array}{lllllll}
 2\ell  -1=ab\\
 \ell  ^2 -6=ac  , \,\, a,b,c\in\NN
\end{array}\right. \II \left\{ \begin{array}{lllllll}
 2\ell   =ab+1\\
 4\ell  ^2 -24=4ac= a^2b^2+2ab -23  
\end{array}\right.$ \\ $\I  a(ab^2+2b-4c)=23, $  hence if $a>1 \I  \left\{ \begin{array}{lllllll}
 a=23\\
(ab^2+2b-4c)=1
\end{array}\right. \I$ \\ $ 23b^2+2b-4c=1$, and so $b=2q+1$, for some $q \in \NN$. Therefore the last equation becomes $92q^2+96q+24=4c$, that is $c=23q^2+24q+6$. \\
Hence  $\left\{ \begin{array}{lllllll}
\ell  =23 q+12 \\
 b=2q+1\\ c= 23q^2+24q+6 .
\end{array}\right.$\quad  $\ell   $ odd $\I q=2k+1 , \, k\in \NN\I \left\{ \begin{array}{lllllll}
\ell  =46k+35 \\
 a=23,\quad b=2q+1\\ c= 23q^2+24q+6 .
\end{array}\right.$
Further, since $e  = (\ell  +2) (2\ell  -1)-(\ell  ^2-6)$, we see that if $\e\in \{35+46k, k\in \NN\}$, then $GCD(e,n_1,n_2)\neq 1$. \\
\\[1mm]
{\it ii.} Since $(\ell  ^2-\ell  -6)=( \ell  -3)(\ell  +2)$, as above we get  
$$\left\{ \begin{array}{lllllll}
  \ell  -3=ab\\
 \ell  ^2+3\ell   +4=ac  , \,\, a,b\,\, odd, c\,\,even
\end{array}\right. \II \left\{ \begin{array}{lllllll}
  \ell   =ab+3\\
  a(ab^2+9 b - c)=-22   
\end{array}\right.$$
hence $a=11, \ell  =11b+3=14+22k,\,k\in \NN$.
\end{proof} 

\begin{lemma}\label{costr} Let $\Gamma':=\{kn_1,k\in[1,\e\,]\}\cup\{ n_2\}\cup\{ s_{p,q}\}\cup\{  r_{p,q}\}\cup\{t_1,\, t_2 \}$. Then:
 \enu
 \item  $ \e n_1-n_2 =(\ell  +1)n_2-(\ell  -1)e=s_{0,\e+1}$   
\item[] 
$\e n_1+e-s_{0,q}\in \{s_{0,q'},\, \, 2\leq q'\leq \ell \}$, for each  $2\leq q\leq \ell$.
 
 \,
 \item $e<n_1<n_2$ are the lowest elements in $\Gamma'$ and $\e n_1$ is the greatest element in $\Gamma'$. 
 \denu
\end{lemma}
\begin{proof}
{\it i.} The equality follows from $\ell n_1=(\ell+2)n_2-(\ell-1)e$, see Construction \ref{asd}. Moreover: \\   
$\ell   n_1+e-s_{0,q}=\ell   n_1+e-(qn_2-(q-2)e)= (\ell  +2-q)n_2-(\ell   -q )e=q'n_2-(q'-2)e,\, $ with $q'=\ell  +2-q\in [2, \e\,]$.\\[2mm]
 {\it ii.  } The first statement follows by a direct check.\, 
  To see that $\e n_1$ is the greatest element, first note that $\e n_1> k n_1$,  if $1\leq k<\e$, and $\e n_1>n_2$. 
  Moreover:  \\
$\e n_1-t_1= -n_1+(\e-1)e=-F+(\e-2)e> 2e-F>0$.\\
$\e n_1-t_2= (\e+1)n_1- \e e>0$.  \\
Now let $\e$ be odd and consider $d =\e n_1-n$, where $n=2e+aF+bG\in \{  s_{p,q}\}\cup\{r_{p',q'}\}$ with $-\e \leq a \leq \e$ and $1\leq b\leq \e+1$:\\ $d =(\e-a) F-bG +(\e-2)e = (\e-2)e+(\e-a-b)F-b(G-F)\geq (\e-2)e-F -(\e+1)(\e^2-2\e-5)= (\e-2)(\e^2+3\e+4)-(2\e-1)-(\e+1)(\e^2-2\e-5)=2\e^2+3\e-2  =n_2 >0$. \\
If $\e$ is even, then $F=\e-3,\, G=\e^2- \e-6\,$ and we get \\
$d=(\e-2)e+(\e-a-b)F-b(G-F)\geq (\e-2)(\e^2+3\e+4)-(\e-3)-(\e+1)(\e^2-2\e-3)=n_2 >0$.\\
\end{proof}

\begin{lemma} \label{costr0} Denote respectively by $F= (2\e-1)$,     $G=(\e^2-6)$, if $\e$ is odd, and by $F= ( \e-3)$, $G=(\e^2-\e-6)$, if $\e$ is even. Then $:$ 
\enu

 \item    $n \in  \{ s_{p,q} \}\cup\{  r_{p,q}\}\I n=  2e+ aF+bG$   with  $  a+b\in[1, \e+1 ],\, a\in[ -\e , \e\,]$, and $b\in[1,  \e+1]$.
\item If $n,n' \in  \{ s_{p,q} \}\cup\{  r_{p,q}\}$, then $n-n'=  aF+bG$\,\, with\,\, $ -\e\leq a+b,b\leq  \e$ and $-2\e \leq a\leq 2\e $.

\item If $ aF+bG=he$,   with   $a,a+b \in[- 2\e-3,2\e+3],\, b\in[-\e-1,2\e+2]$  then $a,b,h$ verify one of the following systems:  \denu
    $ (1) \! \left\{\begin{array}{lrrclllr} 
  a=\nu(\e+2) \\ 
       b=-\nu \\
a+b=\nu(\e+1)\\
 -1\leq \nu\leq 1 \\

 h=\nu, 
 \,{\it if}\, \e \, odd \\
  h=0, \, {\it if} \, \e \, even
  \end{array} 
 \right.    (2)  \!\left\{\begin{array}{lrrclllr}
 
   a = 2  \\ 
  b=     \ell  +1\\
 a+b= \e+3\\
 h=\e-2, 
 \,{\it if}\, \e \, odd \\
  h=\e-3, \, {\it if} \, \e \, even
\end{array} \right. ( 3)     \left\{\begin{array}{lrrclllr} 
 
 a = -  \e  \\ 
 b=   \ell  +2 \\
 a+b=2 \\
 h=\e-3
\end{array} \right. $   $( 4)  \!   \left\{\begin{array}{lrrclllr} 
  
   a = -2\e-2 \\ 
 b=  \ell  +3\\
 a+b=1-\e \\
  h=\e-4,  
\,{\it if}\, \e \, odd \\
  h=\e-3 , \, {\it if} \, \e \, even

\end{array} \right.  $ \\[2mm]
   where  systems $(2),(3),(4)$,   describe the possible situations with 
  $  h> 0  ;$ \, if   $h<0$, then  besides case $(1)$ one has    case  $  (2{\it bis})$ obtained from $(2)$  by changing $a,b,h$ in their  opposite numbers.
  
 \end{lemma}
 
\begin{proof} {\it i}.\, Note that \, $e= -\ell   +(\ell  +2)^2$. Therefore if $\ell$ is  odd: 
$$
\left[\begin{array}{crrllllr}
e  &=& (\ell  +2) (2\ell  -1)&\!-\! & (\ell  ^2-6) =(\ell  +2) F\!-\! G \\
(\ell  -3)\,e  &=&-\ell  \,\,\,(2\ell  -1)&\!+\! &(\ell  ^2-6)(\ell  +2)   
\end{array}\right.
$$ 
 $$ \left[\begin{array}{lllllllr} 
 s_{p,q} &=pn_1\!+\! qn_2\!-\! (p\!+\! q\!-\! 2)e&= 2e&\!+\!&   (2\ell  \!-\! 1)\,p   +  (\ell  ^2\!-\! 6)\,\,\,q  \\

 r_{p,q} & = \ell   n_1+e-s_{p,q}  &=  2e&\!-\!&    (2\ell  \!-\! 1)\,p +  (\ell  ^2\!-\! 6)(\ell  \!+\! 2\!-\! q) \\
& &= 2e&\!+\!&  (2\ell  \!-\! 1)\,p' + (\ell  ^2\!-\! 6)\,q'   \quad    
\end{array} \right.$$

 while if $\ell$ is even we have
$ \begin{array}{cllllllr}
  (\ell  -3)e&=&-\ell   \,( \ell  -3)+(\ell  ^2-\ell  -6)(\ell  +2)  \\
\end{array} $ 
and
$$ \left[\begin{array}{lllllllr}
 s_{p,q} &=pn_1\!+\! qn_2- (p\!+\! q\!-\! 2)e&=&2e&+&    ( \ell  \!-\! 3)\,p  +   (\ell  ^2\!-\! \e\!-\! 6)\,\,\,q  \\
r_{p,q} & = \ell   n_1+e-s_{p,q}  &=& 2e&-&   ( \ell  \!-\! 3)\, p +(\ell  ^2\!-\! \e\!-\! 6)(\ell  \!+\! 2\!-\! q) \\
&&=&2e&\!+\!&  ( \ell  \!-\! 3) \,p' + (\ell  ^2\!-\! \e\!-\! 6)\,q'  \end{array} \right.$$
where\quad $  \left[\begin{array}{lrlllllr}
 2\leq p +q \leq \e+1 ,& 0\leq p\leq  \e,\quad 1\leq q \leq \e+1\\
  1\leq p'+q'\leq \e ,& \!- \e\! \leq p'\leq 0, \quad1\leq q'\leq \e\!+\!1
  \end{array} \right.$\\[2mm]
 {\it ii.  }  It is immediate by part {\it i}.
 \\[2mm]
 {\it iii.} Let $\e$ be {\it odd}  and  $a(2\e-1)+b(\e^2-6)=he=h[(\ell  +2) (2\ell  -1) \!-\!   (\ell  ^2-6)]$.
    Then\\[2mm]  $ [a-h  (\e+2) +\mu(\e^2-6)](2\e-1)+[b+h-\mu(2\e-1)](\e^2-6)=0, \, \,\forall\,\mu\in\ZZ $, therefore\\[2mm]
       $ \left\{\begin{array}{lrrclllr} 
a-h (\ell+2)=-\mu (\ell  ^2-6)\\
 b+h  = \mu (2\ell  -1), \\
  \end{array} \right. \I 
\left\{\begin{array}{lrrclllr} 
 h = \mu (2\ell  -1)-b \\ 
  a-[  \mu (2\ell  -1)-b](\ell+2)=-\mu (\ell  ^2-6) 
\end{array} \right. \I$
$$
a+b(\e +2)= \mu e. 
$$
If $\e$ is {\it even}, let $a(\e-3)+b(\e^2-\e-6)=a(\e-3)+b (\e-3) (\e+2)= (\e-3)[a +b  (\e+2)]=he$.   Since $\e-3$ divides the first member and $(\e-3,e)=1$ by the previous lemma, it follows that $h=(\e-3)\mu$ for some $\mu$ and then $a +b  (\e+2)=\mu e$. Hence in both cases we obtain the equality  $a +b  (\e+2)=\mu e$.\\[2mm]
Now consider $a +b  (\e+2)=\mu e$;  since $-\e\cdot 1+ (\e+2)(\e+2)=e$,   we get:  
{$\left\{\begin{array}{lrrclllr} 
 a = -\mu \e+\nu (\e+2) \\ 
 b= \mu ( \ell  +2)-\nu\\
 a+b=2\mu +\nu (\e+1)
\end{array} \right. $}\\[2mm] 
By the assumptions  $   a,  a+b\in[-2 \e-3,2\e+3] $, we can assume  $h,\mu\geq 0$\,  
(the cases with $h<0$ can be obtained by changing the values of $a,b,a+b$ found for $h>0,$ with $ b\leq\e+1$  in their opposite). Hence:
\enu
\item[$\mu=0 ,$]   $a,a+b \in[- 2\e-3,2\e+3]\I  (1)\left\{\begin{array}{lrrclllr} 
  
a=\nu(\e+2),\\ 
  b=-\nu,\\
a+b=\nu(\e+1)\\
 h=\nu, 
 \,{\it if}\, \e \, odd,\\
 h=0, \, {\it if} \, \e \, even \end{array} \right. $ with $\,  \nu\in[-1,\,1]$\\
  \item[$\mu>0  $]   $\I-1\leq \nu\leq 1$; in fact  $ a+b\in[- 2\e-3,2\e+3]\I  \nu\leq 1$, and \\ $ -\e+\nu(\e+2)\geq a= -\e+\nu(\e+2)-(\mu-1)\e \geq -2 \e-3 \I \nu\geq -1$ $(\nu=-1\I \mu= 1)$.\\[2mm]
 Moreover if $\nu=0,1 $, then $b= \mu ( \ell  +2)-\nu=\mu(\e+1)+\mu-\nu\leq 2\e+2$ implies $\mu\leq 1$.   Hence   $\nu \in \{-1,0,1\}$, $\mu=1$ and we get the    systems $(2),(3),(4)$ of the thesis.  
 \denu\end{proof}

\begin{lemma}\label{costr1} As above, let $\Gamma':=\{kn_1 \ | \ k\in[1,\e\,]\}\cup\{ n_2\}\cup\{  s_{p,q},  r_{p',q'}\}\cup\{t_1,\, t_2 \}$. Then:
\enu
\item  The elements of $\Gamma' $ are all non-zero $ (mod\, e)$;
\item The elements of $\Gamma'$ have distinct residues $(mod\, e)$;
\item If $m,n,n'\in \Gamma'$, the equality $m+n=n'+\alpha e$ implies either $\alpha>0$ or  $\alpha=0$ and  $n'\in\Gamma'':=\{ n_1+n_2, 2n_2, kn_1 \ | \ 2 \leq k\leq \e\}$;
\denu 
  \end{lemma}
  
\begin{proof}
  {\it  i\,.} 
  In fact   for any element $s\in \Gamma',$ we have $ s=aF+bG+ke$. Then if $s=\lambda e$ for some $\lambda\geq 0$, it follows that $aF+bG=he$ with $h=\lambda -k $. Note that\\[2mm]
  \centerline{ $t_1=(\e+1)F+2e$, \quad $t_2=-F+(\e-1)e$.} Then:\\
$\left[\begin{array}{llllr} 

 s\in\{ s_{p,q}\}\cup\{  r_{p',q'}\}&  aF+bG=(\lambda-2)e&   a+b\in[1, \e+1]&b\in[1,\e+1], \,a\in[-\e,\e\,]\\
 s=t_1&  (\e+1)F=(\lambda-2)e& a+b=a=\e+1&b=0\\
 s=t_2&  F=( \e-\lambda-1)e& a+b=a=1&b=0\\
 s=k n_1&  kF=(\lambda-k)e& a+b=a=k\leq \e&b=0\\
 \end{array}\right.$\\[2mm]
Every case verifies the assumptions of Lemma \ref{costr0}.{\it iii}. Hence we can apply this result: note that in cases   $(1)\dots(4)$ of Lemma \ref{costr0}.{\it iii}, either $b=-\nu\in[-1,1]$  or $b\geq \e+1$, and $b=\e+1\II a+b=\e+3$, $b=-\nu\II a+b=\nu(\e+1)$. Easily we can see that no case is possible.\\[2mm]
  {\it ii}.    
  Let $m,n\in   \Gamma'$ and let $m=n+he, h>0$. Then $m=pF+qG+ke, n=p'F+q'G+k'e\I he=m-n=(p-p')F+(q-q')G+(k-k')e\I aF+bG=(h-k+k')e$:  by recalling the above table and Lemma \ref{costr0}.{\it i,ii.} we list the possible $c_i=m-n$: 
   \\[2mm]
 $ \left[\begin{array}{lll}
 &a+b&b\\
c_1:\,m,n\in\{ s_{p,q}, r_{p',q'}\}&[-\e ,\e\,] ,&  [-\e ,\e ]\\
 c_2:\, m\in\{  s_{p,q},  r_{p',q'}\},n=kn_1 \ \ \ &[ 1-k,\e-k+1] \ \ \ & [1,\e+1]\\
  c_3:\, m\in\{  s_{p,q},   r_{p',q'}\},n=n_2&[0 ,\e ]& [0,\e ]\\
  c_4:\, m\in\{  s_{p,q},  r_{p',q'}\},n=t_1&[-\e ,0]& 1,\e+1]\\
  c_5:\, m\in\{  s_{p,q},   r_{p',q'}\},n=t_2&[2,\e+2  ]& [1,\e+1]\\
  c_6: \, m=kn_1,n=n_2&    k-1  & -1\\
  c_{7}:\, m=kn_1,n=t_1&   k-\e-1  & 0\\
  c_{8}:\, m=kn_1,n=t_2&   k+1  & 0\\
 c_{9}:\, m=t_1,n=n_2&   \e  & -1\\
 c_{10}:\, m=t_1,n=t_2&   \e+2  & 0\\
 c_{11}:\, m=t_2,n=n_2&   -2  & -1\\
 \end{array}\right.$ \\[2mm]
As in {\it i}, by Lemma \ref{costr0}.{\it iii}
one can easily see that  no case is possible.\\[2mm]
  {\it iii}.    We denote by:\\
       $p_1 =\left\{\begin{array}{lllll} 
  k n_1&= kF+ke& a+b=a=k\leq \e&b=0,& or \\
  t_1&=  (\e+1)F+2e & a+b=a=\e+1&b=0&or\\
  t_2&= -F+( \e-1)e& a+b=a=-1&b=0\\
 \end{array}\right.$\\
We can write $p_1=aF+0G+\delta e\,\,\delta\in[1,\e\,],\,\, a+b=a \in[-1,\e+1]\setminus\{0\},\,\,b=0$.\\
We divide the elements of $\Gamma'$ in three types:\\[2mm]
$\left[\begin{array}{lllll}

p_1&=aF+ \delta e\quad\delta\in[1,\e\,],& a+b  \in[-1,\e+1] &b=0\\ 
n_2&=  G+ e& a+b=1 &b=1,\,a=0\\
p_3&= aF+bG+2e\in\{  s_{p,q},  r_{p',q'}\}&   a+b\in[1, \e+1]&b\in[1,\e+1], \,a\in[-\e,\e\,]\\

\end{array}\right.$
\\[2mm]
  Denote by $\sigma_1$ any sum   $ p_1+p_1'$:\\[2mm]
  $\sigma_1=\left[\begin{array}{llllr} 
  k n_1+hn_1& = (k+h)F+(k+h)e& a+b\in[2,2\e\,]&b=0\\
  k n_1+t_1&  = (\e+1+k)F+(2+k)e & a+b=a=(\e+1+k) &b=0\\
 k n_1+t_2& =(k-1)F+(k+ \e-1)e& a+b=a=k-1&b=0\\
 2t_1&=(2\e+2)F+4e&a+b=2\e+2&b=0\\
 t_1+t_2&=\e F+(\e+1)e& a+b=\e&b=0\\
 2t_2&=-2F+(2\e-2)e&a+b=-2&b=0\\

 \end{array}\right.$\\[2mm]
 Denote by $\sigma_2$ the sum   $ p_1+n_2$:\\
$\sigma_2=\left[\begin{array}{llllr} 
 k n_1+n_2 & = kF+1G+(k+1) e& a+b=k+1 &a=k&b=1\\
 t_1+n_2& =(\e+1)F+1G+3e & a+b=\e+2&a=\e+1 &b=1\\
  t_2+n_2& =-F+1G+ \e \, e& a+b=0 &a=-1&b=1\\
\end{array}\right.$\\[2mm]
  Let $ p_3,p_3'\in\{s_{p,q},r_{p,q}\}$, where $p_3=a'F+b'G+2e,\,\,p'_3 =a''F+b''G+2e$, and denote by $\sigma_3=aF+bG+\beta e$ \,\, any sum $p 
 _3+p_1, p_3+n_2, p_3+p'_3:$\\
  $\sigma_3=
 \left[\begin{array}{llllr} 
  p_3+ k n_1  &  =(k+a')F+b'G+(k+2)e&   a+b\in[ k+1,  \e+k+1] &b\in[1,\e+1] \\
p_3+t_1 &=(\e+1+a') F+b'G+4e& a+b \in[\e+2,2\e+2]&b\in[1,\e+1] \\
p_3+t_2&=(a'-1)F+b'G+( \e+1)e& a+b \in[0,\e ]&b \in[1,\e+1] \\
p_3+n_2    & = a'F+(b'+1)G +3e&a+b\in[2,\e+2]&b\in[2,\e+2]\\
p_3+p'_3&= (a'+a'')F+(b'+b'')G+4e& a+b\in[2,2\e+2]&b\in[2,2\e+2] \\
 \end{array}\right.$\\[2mm]
  In conclusion we can write:\\[2mm]
 $\left[\begin{array}{llllll} 
\sigma_1&=aF+\lambda e,&  \lambda\in[2,2\e  ] &  a=a+b\in[-2,2\e+2 ]  &b=0\\
\sigma_2&=aF+G+\mu e,&\mu\in[2,\e+1]&a+1=a+b \in[0,\e+2]&b=1\\
 2n_2&=2G+2 e,&\mu=2&a =0&b=2\\
\sigma_3 &=aF+bG+\nu e,&\nu\in[3,\e+2]& a\in[-2\e,2\e]\,a+b\in[0,2\e+2]&  b\in[1,2\e+2]  \\
\end{array}\right.$ \\[2mm]
 $\left[\begin{array}{lllll}

p_1&=aF+ \delta e\quad\delta\in[1,\e\,],& a+b  \in[-1,\e+1] =a &b=0\\ 
n_2&=   G+ e& a+b=1 &a=0\,\,b=1 \,\\
p_3&= aF+bG+2e\in\{ s_{p,q}, r_{p,q}\}&   a+b\in[1, \e+1]& a\in[-\e,\e\,]\,\,b\in[1,\e+1]  \\

\end{array}\right.$\\[2mm]
   Further note that, $\sigma_i=n_2+\alpha e\I \alpha >0$ by Lemma \ref{costr}.{\it ii}, since $2n_1 > n_2$. \\[2mm] 
 Let $m+n=\sigma_i, \,\, n'=p_j, 1\leq i,j\leq 3$, \, assume \, $\sigma_i=p_j+\alpha e$ \, 
 and consider the following table:\\[2mm]
 $\left[\begin{array}{lllll}
 \sigma_1   - p_1  =aF+ue&  u\in[2\!- \e,2\e\!-\!1] & a+b \in[ - \e\!-\!3 ,2\e+3] &b=0\\
  
 \sigma_1   - p_3    =a F+b  G+ue,& u\in[0,2\e\!-\!2] & a+b \in[ - \e\!-\!3,2\e\!+\!1 ] &b\in[ - \e\!-\!1,\!-\!1] \\
 \sigma_2   - p_1  = aF +1  G+ue,&u\in[2\!-\!\e,\e\,]&a+b \in[ - \e\!-\!1,\e+3 ]&b=1 \\
 
 \sigma_2  - p_3  = aF+b  G+ue,&u\in[0, \e\!-\!1 ]&a+b \in[ - \e\!-\!1, \e+1 ]&b\in[-\e ,0] \\
2n_2   - p_1  = aF+2  G+ue,&u\in[2\!- \e,1 ]&a+b \in[ - \e\!+\!1,  3 ]&b=2 \\
2n_2   - p_3  = aF+b  G+0e,&u=0&a+b \in[ - \e\!+\!1,1 ]&b\in[- \e\!+\!1,1 ] \\
 \sigma_3  - p_1    = aF+b  G+ue,&u\in[3\!- \e,\e+ 1 ]&a+b \in[ - \e\!-\!1,2\e+3 ]& b\in[1,2\e+2 ] \\

 \sigma_3  - p_3    = aF +b G+ue,&u\in[1, \e  ]&a+b \in[ - \e\!-\!1,2\e\!+1\! ]&b\in[  - \e,2\e\!+\!1 ] \\
\end{array}\right.$
\\[2mm]
   \centerline{$\sigma_i-p_j=aF+bG+ue=\alpha e\II aF+bG=he,\,\, \,h=\alpha-u \,\,( \alpha=h+u)  \qquad(*)$}
   \\[2mm]
     To prove that, in all possible cases, either $\alpha >0$ or $\alpha=0$ and  $p_j\in\Gamma''$, we can apply Lemma \ref{costr0}.{\it iii},  since the integers $a,b,a+b$ verify the required assumptions.\, 
  It is straightforward to see that the cases $2n_2-p_1$ and $2n_2-p_3$ are impossible, except when $2n_2=p_3$.\\
  $\bullet$  Case $2n_2=p_3$: this equality means    $  2G+2e=a'F+b'G+2e\I a'F+(b'-2)G=0,\,$ with $a'\in[-\e,\e], b'\in[1,\e+1]$. Hence we are in case (1) of   Lemma \ref{costr0}.{\it iii}, with $h=0$ and so either   $b'-2=0, a=0, p_3=2n_2$, or $\e$ even, $b'-2=\pm 1, a= \pm(\e+2)$, that are impossible.   \\
 $\bullet$  Case $\sigma_1-p_1$: {  we have $b=0$ and then $a+b=a=h=0$, by Lemma \ref{costr0}.{\it iii}. If $p_1 \notin \Gamma''$, it is easy to see that $p_1=t_1$ and thus $\sigma_1=(\e+1)F+(\e+1)e$, because $a=0$. In this case $\alpha=u=\e+1-2>0$. \\
$\bullet$  Case $\sigma_1-p_3$: we have $u\geq 0$ and $b \in [-\e-1,-1]$ thus $h \geq 0$ and $\alpha$ is always positive, except when either $ h=u=0,$ {   $ b=-1,a=\e+2, \,\e$ even, or $a=-2,b=-\e-1,  h\in[2-\e,3-\e]$. In the first case,    it is easy to see that $u=h=0$ implies $\sigma_1=2n_1$. From these cases follows that $a+b \geq \e+1$, but it is impossible when $\sigma_1=2n_1$.} In the second one,  $p_3=(\e+1)G+2e\I \sigma_1= 2t_2=-2F+(2\e-2)e\I u= 2\e-4, \alpha\geq\e-2>0$.  \\
$\bullet$  Case $\sigma_2-p_1$: we have $b=1, a+b=-\e-1$, and $h \in\{-1,0\}$ by Lemma \ref{costr0}.{\it iii}. It follows that $\sigma_2=-F+G+(\e+1) e,\ p_1=(\e+1)F+2e$; then $u=\e+1-2\geq 2$ and $\alpha>0$.\\
$\bullet$  Case   $\sigma_2-p_3 $:  In this case $u \geq 0$ and, according to Lemma \ref{costr0}.{\it iii}, we have $b \in [-1,0]$ that implies $h \geq 0$. Thus we always have $\alpha \geq 0$, except when $u=h=0$; it is easy to see that $u=0$ implies $\sigma_2=n_1+n_2$. If $b=0$, then $a+b=0$ implies $p_3=n_1+n_2 \in \Gamma''$.  \\
$\bullet$  Case  {$\sigma_3-p_1=aF+bG+ue=\alpha e, \quad \sigma_3=a'F+b'G+\nu' e$, \quad$p_1=a''F+\delta e$:} \\
 - if $h\geq \e-2\I \alpha >0$, since $u\in[3-\e,\e+ 1 ]$. \\
 - if $h=\e-3=-u$, we have $ \alpha=0$ and $\nu'-\delta=3-\e$, i.e.
$  \delta= \e+\nu'-3$, that implies $\nu'=3$ and $\delta= \e$, because $\nu' \leq 3$ and $\e \geq \delta$; thus $ p_1=\e n_1 \in \Gamma''$.  \\
 - if $h<\e-3$ and $\e$ is odd, the possible cases are $(1) , (4)$ of Lemma \ref{costr0}.{\it iii}:\par
   In case $(1)$ of Lemma \ref{costr0}.{\it iii}, since $b\geq 1$,  we get $  b=1$, thus $h=-b=\nu$,   $a+b=\nu( 1+\e) $: \\ 
-   $ h=-1,\,\, b=1, a+b= -1-\e;\,$  since $a'+b'\geq 0$, we deduce $a''=\e+1$, hence $p_1=t_1=$\par $(\e+1)F+2e$,   $\sigma_3=p_3'+t_2$ has $\nu'=\e+1$, so $u=\e-1>2$ and then $\alpha=h+\e-1>0$.   \par
In case $(4)$ of Lemma \ref{costr0}.{\it iii}:\\
  - $h=\e-4$   then $\alpha>0$ except $u\in[4-\e,3-\e]$.\par
 If $u=4-\e$, then $\alpha=0$ and $ \delta= \e+\nu'-4$, thus $\nu'\in[3,4] $ since $\delta\leq \e$. Hence\par   
 $\nu'=4\I\delta=\e \I p_1=\e n_1 \in \Gamma''$.\par
 $\nu'=3\I\delta=\e-1\I p_1=(\e-1)n_1 \in \Gamma''$, because if $p_1=t_2$ then $a+b \in [1,2\e-1]$\par that is a contradiction by   (4) .\par 
  If $u=3-\e$, then $ \delta= \e+\nu'-3$ implies $\nu'=3$ as above, then the possible $\sigma_3$ are $p_3+n_1,$ \par$ p_3+n_2$ and so $b=b'\leq \e+2$, that is incompatible with the value   $b=\e+3$ in $(4)$.\\
- $h<\e-3$   and  $\e$ is even, then $h=0$, $b=-\nu$, $a+b=\nu( 1+\e) $. Since in $\sigma_3-p_1$ we have \par $a+b\geq -\e-1$ then $\nu\geq -1$. Now  $b\geq 1$, implies $\nu=-1, b=1$, and $a+b=-\e-1$.\par Therefore we can proceed as in the case $h=-1,b=1$ treated above when $\e$ is odd.  \\  
   $\bullet$    Case $\sigma_3-p_3$: one always has $\alpha > 0$, except if   $ b=1=u, a+b=-\e-1$ and $\e$ is odd; in this case
  $\sigma_3=a'F+b'G+\nu' e$, with $a'+b'\geq 0$, hence $a+b=-\e-1$ implies that $p_3=a''F+b''G+2e, $ with $a''+b''=\e+1$ and $a'+b'=0$.
  Then $\sigma_3=p'_3+t_2=a'F+b'G+(\e+1)e$ and $u=\e-1>2$,   that is  a contradiction because $u=1$.  }
This proves {\it iii}.
\end{proof}

\begin{proposition} \label{apery} The Ap\'ery set of $S$ is\\[2mm]  
\centerline {$ \{0\}\,\cup \,\Gamma'$ \, with $\Gamma'= \{kn_1 \ | \ k\in[1,\e\,]\}\cup\{ n_2\}\cup\{ s_{p,q}\}\cup\{  r_{p,q}\}\cup\{t_1,\, t_2 \} $ \ } \\[2mm] 
 \centerline{$ Ap_2 =\{2n_1,n_1+n_2,2n_2\}$  \qquad $Ap_k =\{kn_1\},$ \, for \, $3\leq k\leq \e $.\hspace{1.5cm} } \\[2mm] 
Further $\e n_1-e$ is the Frobenius number of $S$ and, with the notation of Theorem \ref{nari},\\[1mm] 
 $A=\{0,n_2, s_ {0,\ell+1} \}\cup \{kn_1 \ | \ k\in[1,\e\,] \, \},$ \quad
  $B=\{s_{p,q} \ | \ (p,q)\neq  (0,\e+1 )\}\cup \{r_{p,q}\}\cup\{t_1,t_2\}.$ 
\end{proposition}

\begin{proof}  We have that $|\{0\}\,\cup \,\Gamma'|=e$, as follows by Construction \ref{asd};  further these elements are all distinct $mod \,e$   by the previous lemma.  
Now we want to  prove that for $s,s_1, \dots, s_r\in \Gamma'$, the equality $s=s_1+...+s_r+\beta  e,\,r\geq 2$, with $\beta  \geq 0$, is impossible or implies $\beta =0 $ and $s\in\Gamma''= \{ n_1+n_2,2n_2\}\cup \{kn_1 \ | \ k\geq 2\}$.\\
By the first assertion, there exists $s'\in \Gamma'$ such that $s_1+s_2=s'+\beta' e$, where $ \beta'\geq 0$ by Lemma \ref{costr1}. Therefore if $r\geq 3$, we have $s=s'+s_3+...+s_r+(\beta +\beta') e$; by iterating and by Lemma \ref{costr1}.{\it iii}, we deduce that the unique possible case is $s\in \Gamma '' $ with $\beta=0$.  
  Since the minimal generators of $S$ are in $\Gamma'$, this proves that for every element $s\in \Gamma'$ we have $s-e\notin S$ and, by Lemma \ref{costr1}.{\it ii} this means that $Ap={0} \cup \Gamma'$; moreover the elements in $\Gamma'$ of order greater than $1$ are in $\Gamma''$.\par  
Now we show that  $Ap_2=\{2n_1,n_1+n_2,2n_2\}$. 
In fact, recalling that, by Lemma \ref{costr}.{\it ii}, $n_1<n_2$ are the smallest elements in $\Gamma'$, it follows that $ord(2n_1)=ord(n_1+n_2)=2$. On the other hand  $ord (2n_2)=2$, because   in the proof of Lemma \ref{costr1}.{\it iii}  we proved that $\sigma_i=2n_2\II \sigma_i=n_2+n_2$. \par
  Moreover $ord(kn_1)=k$ because $n_1$ is the smallest element of $S$ and then $Ap_k=\{ k n_1\}$.\par  
 Finally, it follows from Lemma \ref{costr}.{\it ii} that $(\e n_1-e)$ is the Frobenius number of $S$ and, by construction and by Lemma \ref{costr}.{\it i}, we have:\\ 
$A=\{0,n_2, s_{0,\ell+1}\}\cup \{kn_1 \ | \ k\in[1,\e\,] \, \}, \quad$ 
$B=\{s_{p,q} \ | \ q\neq  \e+1 \}\cup \{r_{p,q}\}\cup\{t_1,t_2\}.$
\end{proof}

  \noindent \textbf{Acknowledgments.} The authors would like to thank A. Sammartano for pointing their attention to \cite{P}.

\end{document}